\documentclass[10pt]{amsart}

\usepackage{amsmath, amsfonts, amsthm, amssymb, graphicx, fullpage, enumerate, subcaption, float, caption, hyperref}
\usepackage{scalerel}
\usepackage{boldline} 
\usepackage{makecell} 
\usepackage{longtable}  
\usepackage{array}     
\usepackage{lingmacros}  
\usepackage{tree-dvips} 
\usepackage{amsmath}
\usepackage{amsthm}  
\usepackage{amsfonts} 
\usepackage{graphicx}
\usepackage{algorithm}
\usepackage{algpseudocode} 
\graphicspath{ {./images/} } 
\newtheorem{theorem}{Theorem}[section]         

\newtheorem{corollary}[theorem]{Corollary}

\newtheorem{proposition}[theorem]{Proposition}
    
\newtheorem*{main theorem}{Main Theorem}

\theoremstyle{remark}       
   
\theoremstyle{definition}  
\newtheorem{definition}[theorem]{Definition}

\def\N{\mathbb{N}}

\def\Z{\mathbb{Z}}

\begin{document}
\title{Computations and Observations on Congruence Covering Systems} 
\author{Raj Agrawal, \quad Prarthana Bhatia, \quad Kratik Gupta, \quad Powers Lamb, \\ Andrew Lott, \quad Alex Rice, \quad Christine Rose Ward}
 
\begin{abstract} A \textit{covering system} is a collection of integer congruences such that every integer satisfies at least one congruence in the collection. A covering system is called \textit{distinct} if all of its moduli are distinct. An expansive literature has developed on covering systems since their introduction by Erd\H{o}s. Here we provide a full classification of distinct covering systems with at most ten moduli, which we group together based on two forms of equivalence. As a consequence, we determine the minimum cardinality of a distinct covering system with all moduli exceeding $2$, which is $11$.
\end{abstract}

\address{Department of Mathematics, Millsaps College, Jackson, MS 39210}
\email{agrawr@millsaps.edu} 
\email{bhatip@millsaps.edu}
\email{guptak1@millsaps.edu}  
\email{lambps@millsaps.edu}

\email{riceaj@millsaps.edu}
 
\email{wardcr@millsaps.edu}

\address{Department of Mathematics, University of Georgia, Athens, GA 30602}

\email{andrew.lott@uga.edu}  

\maketitle 
\setlength{\parskip}{5pt}   
 
\section{Introduction} 

It may seem a rather uninspiring fact that every integer is congruent to either $0$ or $1$ modulo $2$. However, if we impose additional requirements on the residues or moduli, the game of accounting for all integers with a collection of congruences becomes much more interesting, leading to the following family of definitions. 

\begin{definition}  For $a,m\in \Z$ with $m\geq 2$, the \textit{congruence class} $a (\text{mod }m)$ is the set of all integers congruent to $a$ modulo $m$.  A \textit{system of congruences} is a collection of congruence classes $ \{r_1 (\text{mod }m_1), \dots,r_k (\text{mod }m_k)\}$. Such a collection is called a \textit{covering system} if every integer $n$ satisfies $n \equiv r_i \ (\text{mod }m_i)$ for some $1\leq i \leq k$. A covering system is called \textit{distinct} if all the moduli are distinct, and \textit{minimal}  if all of the congruence classes are needed to cover the integers. In other words, in a minimal covering system, if one were to remove any one of the congruence classes, the remaining classes would \textit{not} form a covering system.\end{definition}

Covering systems were introduced by Erd\H{o}s \cite{erdos} as a component of his proof of a conjecture of Romanoff that there exists an arithmetic progression of odd numbers, none of which take the form $2^k+p$ for $k\in \N$ and $p$ prime. Specifically, his proof utilized the distinct covering system \begin{equation} \label{origCS} \{0(\text{mod }2), \ 0(\text{mod }3), \ 1(\text{mod }4), \ 3(\text{mod }8), \ 7(\text{mod }12), \ 23(\text{mod }24) \}.\end{equation}  Inspired by a possible generalization of his proof, Erd\H{o}s conjectured that there exist distinct covering systems with arbitrarily large minimum modulus, which became a coveted open problem. Nielsen \cite{nielsen} discovered a distinct covering system with minimum modulus $40$, and was the first to entertain in writing the possibility of a negative resolution to Erd\H{o}s's conjecture. To date, the largest known minimum modulus of a distinct covering system is 42, discovered by Owens \cite{Owens}. Nielsen's suspicion was proven reality by Hough \cite{hough} in 2015, who showed that the minimum modulus of a distinct covering system is at most $10^{16}$. This upper bound has since been lowered all the way to $616000$ in work of Balister, Bollobas, Morris, Sahadrabudhe, and Tiba \cite{BBMST}.  

Another famous question, due to Erd\H{o}s and Selfridge, is the existence or nonexistence of a distinct covering system with all odd moduli. This question remains open, but for recent progress the interested reader can refer to \cite{BBMST}, \cite{harrington}, \cite{HHM}, and \cite{HSW}. Here we have only scratched the surface of the massive covering system literature, focusing on recent results. For a more complete survey and list of references, particularly regarding less recent work, refer to \cite{PS} and \cite{zen}. 

The covering system \eqref{origCS} can be modified by replacing the final three congruences with two, yielding  \begin{equation} \label{5CS} \{0(\text{mod }2), \ 0(\text{mod }3), \ 1(\text{mod }4), \ 1(\text{mod }6), \ 11(\text{mod }12)\}. \end{equation} It is noted throughout the literature that \eqref{5CS} is a distinct covering system of minimum cardinality. In Section \ref{obs}, we provide explicit case analysis to confirm this assertion, and also address some follow-up questions.  These discussions lead us toward two natural forms of equivalence for distinct covering systems, which we use in our later computational classification efforts. 
 
In Section \ref{compsec}, we follow techniques developed by Jenkin and Simpson \cite{JS}, who found distinct covering systems with all composite moduli and minimum cardinality (which turns out to be $20$), to extend our reach with algorithms written in Python. Specifically, we classify all distinct covering systems with at most ten moduli, and group them together based on the two forms of equivalence discussed in Section \ref{obs}. In an abridged presentation of our findings, we include representatives of all equivalence classes of distinct minimal covering systems with at most seven moduli, as well a list of the sets of eight moduli that yield distinct minimal covering systems, with counts of how many equivalence classes arise from each set. We also provide summary data for $k=9,10$, and include a link to our full lists and code for the interested reader.

A notable finding in our classification is that all distinct covering systems with at most ten moduli have minimum modulus $2$. When making his aforementioned conjecture on the minimum modulus of distinct covering systems, Erd\H{o}s \cite{erdos} provided a distinct covering system with minimum modulus $3$, which utilizied $14$ moduli, the divisors of $120$ that are greater than $2$. In \cite{erdos2}, he guessed that this system had minimum cardinality amongst distinct covering systems whose moduli are all greater than $2$, but this was found to be incorrect by Krukenberg \cite{Kruk}, whose thesis included the $11$-modulus distinct covering system \begin{equation}\label{krukcs}  \{[2, 3], \ [0, 4], \ [1, 6], \ [2, 8], \ [0, 9], \ [3, 12], \ [6, 16], \ [3, 18], \ [6, 24], \ [33, 36], \ [46, 48]\}. \end{equation} Here and for the remainder of the paper we use the shorthand notation $[r,m]$ for the congruence class $r(\text{mod }m)$.  Our classification efforts combine with \eqref{krukcs} to yield the following conclusion.

\begin{proposition} \label{mainprop} The minimum cardinality of a distinct covering system with all moduli exceeding $2$ is $11$.
\end{proposition} 

Building from \cite{Kruk}, Dalton and Trifonov \cite{DT} have recently investigated the minimum least common multiple of a distinct covering system with a given minimum modulus. However, this does not necessarily correspond to distinct covering systems of minimal cardinality, and Proposition \ref{mainprop} is, to our knowledge, the first result of its specific type. We conclude with a complete characterization of distinct covering systems with exactly $11$ moduli, all exceeding $2$, all of which have the same set of moduli as \eqref{krukcs}.
 
%Further, \eqref{krukcs} is, up to affine equivalence, the unique such covering system of cardinality $11$. 

\section{Preliminaries} \label{obs} 

We begin with some standard facts that are helpful in determining when a system of congruences is or is not a covering system. 

\begin{proposition}\label{lcmprop} A system of congruences $S=\{r_1(\textnormal{mod }m_1),\dots,r_k(\textnormal{mod }m_k)\}$ is a covering system if and only if it covers a member of every congruence class modulo $M=\textnormal{lcm}(m_1,\dots,m_k)$. 
\end{proposition}

The following definition and proposition are particularly helpful when ruling out a set of moduli from potentially producing a covering system.

\begin{definition} For a system of congruences $S=\{r_1(\textnormal{mod }m_1),\dots,r_k(\textnormal{mod }m_k)\}$, let $R(S)=\sum_{i=1}\frac{1}{m_i}$.

\end{definition}

\begin{proposition}\label{rec1} If $S$ is a covering system, then $R(S)\geq 1$, with equality holding if and only if $S$ is exact. 
\end{proposition}

\subsection{Distinct Covering Systems}

As mentioned in the introduction, a wide variety of surveys, articles, and books (see \cite{JS}, \cite{nielsen}, and \cite{PS} for just a few examples) mention that \eqref{5CS}, or another covering system with the same set of moduli, is a distinct covering system of minimum cardinality. Presumably the verification of this fact has been consistently left as a pleasing exercise for the reader, which we carry out below after developing some useful notation.

\begin{definition} For $k\in \N$, we let $\mathcal{C}_k$ denote the collection of all distinct minimal covering systems with exactly $k$ moduli. 
\end{definition}

\begin{proposition} \label{case} The collection $\mathcal{C}_k$ is empty for $k\leq 4$. 
\end{proposition}

\begin{proof} Suppose $S=\{r_1(\textnormal{mod }m_1),\dots,r_k(\textnormal{mod }m_k)\}$ with $m_1<\cdots < m_k$ is a distinct minimal covering system. We begin by quickly ruling out $k\leq 3$.  By Proposition \ref{rec1}, since $1/2+1/4+1/5<1$, the only candidates for $k=3$ have $m_1=2$, $m_2=3$. However, by the Chinese remainder theorem, choosing $r_1,r_2$ leaves two missing classes modulo $6$, which cannot be covered by a single class modulo $m\geq 4$.  

\noindent Now we consider $k=4$. By Proposition \ref{rec1}, since $1/3+1/4+1/5+1/6<1$, we must have $m_1=2$. We perform case analysis as follows:

\noindent \textbf{Case 1:} $m_2\geq 5$.

\begin{enumerate}[(a)]

\item  $m_4\geq 8$. Impossible by Proposition \ref{rec1} as $1/2+1/5+1/6+1/8<1$.

\item  $(m_2,m_3,m_4)=(5,6,7)$. Choosing $r_1,r_3$ leaves exactly two uncovered classes modulo $6$. Then, by the Chinese remainder theorem, choosing $r_2$ leaves eight missing classes modulo $30$, which cannot be covered by a single class modulo $7$. 

\end{enumerate}

\noindent \textbf{Case 2:} $m_2= 4$. The choice $r_1,r_2$ leaves one uncovered class modulo $4$.

\begin{enumerate}[(a)]

\item $m_3\geq 8$. Impossible by Proposition \ref{rec1} as $1/2+1/4+1/8+1/9<1$.

\item $m_3\in\{5,7\}$. By the Chinese remainder theorem, choosing $r_3$ leaves either four missing classes modulo $20$ or six missing classes modulo $28$, which cannot be covered by a single class modulo $m\geq 6$. 

\item $m_3=6$. Choosing $r_3$ leaves two uncovered classes modulo $12$, which cannot be covered by a single class modulo $m\geq 7$.

\end{enumerate}

\noindent \textbf{Case 3:} $m_2= 3$. By the Chinese remainder theorem, choosing $r_1,r_2$ leaves two missing classes modulo $6$.

\begin{enumerate}[(a)]

\item $m_3\geq 6$. Impossible since $1/6+1/7<1/3$, so the last two classes cannot cover two classes modulo $6$.

\item  $m_3=5$. By the Chinese remainder theorem, choosing $r_3$ leaves eight missing classes modulo $30$, which cannot be covered by a single class modulo $m\geq 6$. 

\item $m_3=4$. Choosing $r_3$ leaves two uncovered classes modulo $12$, which are incongruent modulo $3$ and hence incongruent modulo $6$, so they cannot be covered by a single class modulo $m\geq 5$. \end{enumerate}\end{proof}
Since \eqref{5CS} does indeed have minimum cardinality amongst distinct minimal covering systems, it is natural to ask whether it is, in any sense, unique in this regard. In the most literal sense, we quickly see this to not be the case, as one can add any fixed number to each residue, or take the negative of every residue, yielding a total of $24$ technically different distinct covering systems with  moduli $\{2,3,4,6,12\}$.
 
The following definition and proposition generalize this observation that a covering system immediately spawns a family of related covering systems via simple transformations. The proposition follows from a more general result of Jones and White \cite{JW}.

\begin{definition} For a system of congruences $S=\{r_1(\textnormal{mod }m_1),\dots,r_k(\textnormal{mod }m_k)\}$ and $a,n\in \Z$, we define $$aS+n=\{ar_1+n(\text{mod }m_1),\dots,ar_k+n(\text{mod }m_k)\}.$$
\end{definition}

\begin{proposition}\label{affine} Suppose $S=\{r_1(\textnormal{mod }m_1),\dots,r_k(\textnormal{mod }m_k)\}$ is a system of congruences, and let $M=\textnormal{lcm}(m_1,\dots,m_k)$. Suppose further that $a,n\in \Z$ with $\gcd(a,M)=1$. Then, $S$ is a covering system if and only if $aS+n$ is a covering system. 
\end{proposition}

\noindent In particular, Proposition \ref{affine} induces an equivalence relation on $\mathcal{C}_k$ where $S$ is equivalent to $aS+n$ for all $a,n\in \Z$ with $\gcd(a,M)=1$, which we refer to as \textit{affine equivalence}.

Since we know that the minimum cardinality of a distinct minimal covering system is $5$, a natural question is whether distinct minimal covering systems exist for \textit{all} larger cardinalities. Note that the insistence that the covering systems be minimal prevents us from simply tacking on additional congruences to existing covering systems. The following definition and proposition demonstrate a quick and elementary way of using an existing distinct minimal covering system to produce a new distinct minimal covering system with exactly one additional congruence. 

\begin{definition}For a system of congruences $S=\{r_1(\textnormal{mod }m_1),\dots,r_k(\textnormal{mod }m_k)\}$, we define $$\delta(S)=\{1(\text{mod }2),2r_1(\text{mod }2m_1),\dots,2r_k(\text{mod }2m_k)\}.$$ A system $S\in \mathcal{C}_k$ is called \textit{$\delta$-primitive} if $S,S+1 \notin\delta(\mathcal{C}_{k-1})$.
\end{definition}

\begin{proposition} \label{k} For a system of congruences $S$, $\delta(S)\in \mathcal{C}_{k+1}$ if and only if $S\in \mathcal{C}_k$. 
\end{proposition}

\begin{proof} Suppose $S=\{r_1(\textnormal{mod }m_1),\dots,r_k(\textnormal{mod }m_k)\}$ is a system of congruences. To say that $S\in \mathcal{C}_k$ is to say three things: $S$ is a covering system, $S$ has exactly $k$ congruences with all distinct moduli, and if any of the congruences were removed from $S$, it would no longer be a covering system. We will show that each of these properties hold if and only if the analogous properties hold for $\delta(\mathcal{C}_k)$, with $k$ replaced by $k+1$. 

\noindent Since $S'=\{2r_{1} (\text{mod }2m_{1}), \dots,  2r_k (\text{mod }2m_{1})\}$ consists entirely of even integers, while $1(\text{mod }2)$ is the set of all odd integers, we see that $\delta(S)=\{1(\text{mod }2),  2r_1  (\text{mod }2m_{1}), \dots,  2r_k (\text{mod }2m_{1})\}$ is a covering system if and only if   $S'$ covers all even integers, which is equivalent to $S$ covering all integers. 

\noindent This equivalence can also be applied to compare $S \setminus \{r_{i}(\text{mod }m_i)\}$ and $\delta(S)\setminus \{2r_{i}(\text{mod }2m_i)\}$ for each $1\leq i \leq k$. The former fails to be a covering system if and only if the latter does (and $\delta(S)\setminus \{1(\text{mod }2)\}$ always fails to be a covering system), hence $S$ is minimal if and only if $\delta(S)$ is minimal.

\noindent Finally, for distinctness, $m_i\neq m_j$ if and only if $2m_{i} \neq 2m_{j}$, and $m_1,\dots,m_k>1$ implies $2m_1,\dots,2m_k>2$, so the $k$ moduli of $S$ are all distinct if and only if the $k+1$ moduli of $\delta(S)$ are all distinct. \end{proof}

By starting with any member of $\mathcal{C}_5$, say \eqref{5CS}, and iteratively applying $\delta$, we have the following corollary.

\begin{corollary} The collection $\mathcal{C}_k$ is nonempty for all $k\geq 5$.
\end{corollary}

The map $\delta$ induces an equivalence relation on $\mathcal{C}=\bigcup_{k=5}^{\infty}\mathcal{C}_k$, where two covering systems are equivalent if one can be obtained from the other by applying $\delta$ some number of times. In other words, the $\delta$-equivalence classes are determined by the $\delta$-primitive covering systems, which each spawn an infinite family of distinct minimal covering systems via iteration of $\delta$.

\section{Computations} \label{compsec}

Following the lead of Jenkin and Simpson \cite{JS}, we conduct a computation to classify all distinct minimal covering systems with at most $10$ moduli. We call a list of moduli $\{m_1,\ldots,m_k\}$ \textit{good} if there exist residues $\{r_1,\ldots,r_k\}$ such that $\{r_1(\textnormal{mod }m_1),\dots,r_k(\textnormal{mod }m_k)\}$ is a covering system and \textit{bad} otherwise. Our first goal is to create a manageable list which contains all lists of good moduli with cardinality at most $10$. The following proposition from \cite{JS} is crucial to our approach:
\begin{proposition}\label{lcmlist}
If $S=\{r_1(\textnormal{mod }m_1),\dots,r_k(\textnormal{mod }m_k)\}$ is a minimal covering system and $\prod_{i=1}^t p_i^{a_i}$ is the prime factorization of $\textnormal{lcm}(m_1,\dots,m_k)$, then 
\[\sum_{i=1}^t a_i(p_i-1)+1\leq k.\]
\end{proposition} 
Since we are only concerned with covering systems of cardinality at most $10$, this propsition reduces our search to a finite list of potential least common multiples, say $L$. As previously known, and as shown explicitly in Proposition \ref{case}, there are no distinct covering systems of cardinality less than $5$, so our goal is to produce a list which contains all good lists of moduli with cardinality at least $5$ and at most $10$, each of which has least common multiple in $L$. To do this, we employ an algorithm which takes a list of integers $M$ as input, and builds a list that contains, for each $m\in M$, all subsets of distinct divisors of $m$ of cardinality $5\leq i \leq 10$. To optimize efficiency,  we run the algorithm with $M=L'$, where $L'\subseteq L$ is a subset of minimal size with the property that every element of $L$ divides an element of $L'$. Now that we have an initial list of potentially good lists of moduli, we are ready to further whittle down the search. Specifically, we use a powerful algorithm outlined in Section 3 of \cite{JS}, which takes as input a list of moduli and returns `bad' or `don't know'.  Although the algorithm cannot always detect whether a list is bad, it does so often enough to greatly reduce the search space. 

We check each of the remaining lists of moduli in a more brutal manner.  For each list $m_1<\cdots<m_k$, we create a list of systems $\{r_1(\textnormal{mod }m_1),\dots,r_k(\textnormal{mod }m_k)\}$ with the following property: for $1\leq i<j\leq k$, if $m_i\mid m_j$, then $r_j\not\equiv r_i\pmod {m_i}$. Otherwise, $r_j(\text{mod }m_j)$ would be entirely contained in $r_i(\text{mod }m_i)$, and the system would not be minimal. The systems are also chosen to guarantee that there is at least one representative from every affine equivalence class. 

Equipped with a list of potential systems for a specific list of moduli, we check whether each system covers $\mathbb{Z}$. We accomplish this by exploiting Proposition \ref{lcmprop} and checking whether the system covers $\{0,1,\dots,M-1\}$, where $M$ is the least common multiple of the moduli.

%\begin{algorithm}
%\begin{algorithmic}
%\State{\textbf{Input}: System}
%\State{\textbf{Compute} $l = \text{LCM of the moduli of System}$}
%\For{i in $[0,l-1]$}:
%	\If{$\text{System does not cover } i$}: 
%		\State{\text{Return False}}
%	\EndIf
%\EndFor
%\State{\text{Return True}}
%\end{algorithmic}
%\end{algorithm}

Left with a list of covering systems, we use straightforward algorithms to check each for minimality and reduce the list so that there are unique representatives for each   $\delta$-primitive affine equivalence class. The following table contains an abridged version of the results of our computations, while our complete lists of data, and all annotated code, are available at \url{https://github.com/andrewlott99/Kinnaird22_CoveringSystems}.

\begin{table}[H]

\caption{A single representative from each $\delta$-primitive affine equivalence class in $\mathcal{C}_k$ for $k=5,6,7$, a complete list of all sets of moduli for $\mathcal{C}_8$, and summary data for $\mathcal{C}_9$ and $\mathcal{C}_{10}$.}
\label{CSdata}

\begin{tabular}{|| l | l  l ||}
\hline
\hline
$k$ & \multicolumn{2}{| c ||}{$\delta$-primitive affine equivalence classes in $\mathcal{C}_k$}\\
\Xhline{0.8pt}
$5$ & $\{[1, 2], [1, 3], [2, 4], [2, 6], [0, 12]\}$ & \ \\
\hline

$6$ & $\{[1, 2], [1, 3], [2, 4], [2, 6], [4, 8], [0, 24]\}$, & $\{ [1, 2], [1, 3], [2, 4], [4, 8], [8, 12], [0, 24]\}$, \\ \ & $\{[1, 2], [1, 3], [2, 6], [4, 8], [6, 12], [0, 24] \}$ & \ \\
\hline 

$7$ & \multicolumn{2}{| l ||}{15 equivalence classes, 11 sets of moduli:} \\ 

\ & $\{[1, 2], [1, 3], [2, 4], [2, 6], [4, 8], [8, 16], [0, 48] \}$, & $\{[1, 2], [1, 3], [2, 4], [2, 6], [3, 9], [6, 18], [0, 36] \}$, \\ 

\ & $\{[1, 2], [1, 3], [2, 4], [2, 6], [6, 9], [12, 18], [0, 36] \}$, & $\{[1, 2], [1, 3], [2, 4], [2, 6], [8, 16], [12, 24], [0, 48] \}$, \\ 

\ & $\{[1, 2], [1, 3], [2, 4], [4, 8], [8, 12], [8, 16], [0, 48] \}$, & $\{ [1, 2], [1, 3], [2, 4], [4, 8], [8, 16], [8, 24], [0, 48]\}$, \\ 

\ & $\{[1, 2], [1, 3], [2, 4], [3, 9], [8, 12], [6, 18], [0, 36] \}$, & $\{[1, 2], [1, 3], [2, 4], [6, 9], [8, 12], [12, 18], [0, 36]\}$,\\

\ & $\{[1, 2], [1, 3], [2, 4], [8, 12], [8, 16], [12, 24], [0, 48] \}$, & $\{[[1, 2], [1, 3], [2, 6], [4, 8], [6, 12], [8, 16], [0, 48] \}$, \\

 \ & $\{ [1, 2], [1, 3], [2, 6], [3, 9], [6, 12], [6, 18], [0, 36]\}$, & $\{ [1, 2], [1, 3], [2, 6], [6, 9], [6, 12], [12, 18], [0, 36]\}$, \\ 

\ & $\{[1, 2], [1, 3], [2, 6], [6, 12], [8, 16], [12, 24], [0, 48] \}$, & $\{ [1, 2], [2, 4], [2, 6], [3, 9], [4, 12], [6, 18], [0, 36]\}$, \\ 

\ & $\{[1, 2], [2, 4], [2, 6], [6, 9], [4, 12], [12, 18], [0, 36] \}$  & \ \\

\hline

$8$ &  \multicolumn{2}{| l ||}{85 equivalence classes, one or two with each of the following 50 sets of moduli:} \\ 

\ & \multicolumn{2}{| l ||}{ $\{2, 3, 4, 6, 8, 9, 18, 72 \}$,  $\{ 2, 3, 4, 6, 8, 9, 36, 72\}$,  $\{ 2, 3, 4, 6, 8, 16, 32, 96\}$,  $\{2, 3, 4, 6, 8, 18, 36, 72 \}$,} \\

\ & \multicolumn{2}{| l ||}{ $\{2, 3, 4, 6, 8, 32, 48, 96 \}$,  $\{2, 3, 4, 6, 9, 18, 24, 72 \}$,  $\{2, 3, 4, 6, 9, 24, 36, 72 \}$,  $\{ 2, 3, 4, 6, 16, 24, 32, 96\}$,}\\

\ & \multicolumn{2}{| l ||}{ $\{2, 3, 4, 6, 18, 24, 36, 72 \}$,  $\{ 2, 3, 4, 6, 24, 32, 48, 96\}$, \ $\{2, 3, 4, 8, 9, 12, 18, 72 \}$,  $\{2, 3, 4, 8, 9, 12, 36, 72 \}$,} \\

\ & \multicolumn{2}{| l ||}{ $\{2, 3, 4, 8, 9, 18, 24, 36 \}$,  $\{2, 3, 4, 8, 9, 18, 24, 72 \}$, $\{ 2, 3, 4, 8, 9, 24, 36, 72\}$,  $\{ 2, 3, 4, 8, 12, 16, 32, 96\}$,} \\

\ & \multicolumn{2}{| l ||}{ $\{ 2, 3, 4, 8, 12, 18, 36, 72\}$,  $\{2, 3, 4, 8, 12, 32, 48, 96 \}$,  $\{2, 3, 4, 8, 16, 24, 32, 96 \}$,  $\{2, 3, 4, 8, 16, 32, 48, 96 \}$,} \\

\ & \multicolumn{2}{| l ||}{ $\{ 2, 3, 4, 8, 18, 24, 36, 72\}$,  $\{ 2, 3, 4, 8, 24, 32, 48, 96\}$,  $\{2, 3, 4, 9, 12, 18, 24, 72 \}$,  $\{2, 3, 4, 9, 12, 24, 36, 72 \}$,} \\

\ & \multicolumn{2}{| l ||}{ $\{2, 3, 4, 12, 16, 24, 32, 96 \}$,  $\{ 2, 3, 4, 12, 18, 24, 36, 72\}$,  $\{2, 3, 4, 12, 24, 32, 48, 96 \}$, $\{ 2, 3, 6, 8, 9, 12, 18, 72\}$,} \\

\ & \multicolumn{2}{| l ||}{ $\{ 2, 3, 6, 8, 9, 12, 36, 72\}$, $\{2, 3, 6, 8, 9, 18, 24, 36 \}$,  $\{ 2, 3, 6, 8, 9, 18, 36, 72\}$,  $\{2, 3, 6, 8, 12, 16, 32, 96 \}$,}\\

\ & \multicolumn{2}{| l ||}{ $\{2, 3, 6, 8, 12, 18, 36, 72 \}$,  $\{2, 3, 6, 8, 12, 32, 48, 96 \}$,  $\{2, 3, 6, 9, 12, 18, 24, 72 \}$,  $\{ 2, 3, 6, 9, 12, 24, 36, 72\}$,} \\

\ & \multicolumn{2}{| l ||}{ $\{ 2, 3, 6, 9, 18, 24, 36, 72\}$,  $\{2, 3, 6, 12, 16, 24, 32, 96 \}$, $\{ 2, 3, 6, 12, 18, 24, 36, 72\}$,  $\{2, 3, 6, 12, 24, 32, 48, 96 \}$,} \\
 
\ & \multicolumn{2}{| l ||}{ $\{2, 4, 6, 8, 9, 12, 18, 72 \}$,  $\{2, 4, 6, 8, 9, 12, 36, 72 \}$,  $\{2, 4, 6, 8, 9, 18, 24, 36 \}$,  $\{2, 4, 6, 8, 9, 18, 24, 72 \}$,} \\

\ & \multicolumn{2}{| l ||}{ $\{2, 4, 6, 8, 9, 24, 36, 72 \}$,  $\{2, 4, 6, 9, 12, 18, 24, 72 \}$,  $\{2, 4, 6, 9, 12, 24, 36, 72 \}$,  $\{2, 4, 8, 9, 12, 18, 24, 36 \}$,} \\

\ & \multicolumn{2}{| l ||}{ $\{2, 4, 8, 9, 12, 18, 24, 72 \}$,  $\{2, 4, 8, 9, 12, 24, 36, 72 \}$} \\
 
\hline

$9$ &  \multicolumn{2}{| l ||}{585 equivalence classes, 248 sets of moduli:} \\ 

\ & \multicolumn{2}{| l ||}{$\bullet$ $1,2,4$ or $6$ equivalence classes for each set of moduli} \\

\ & \multicolumn{2}{| l ||}{$\bullet$ All systems have minimum modulus $2$} \\

\ & \multicolumn{2}{| l ||}{$\bullet$ Maximum moduli: $30,48,60,72,80,108,144,192$} \\

\ & \multicolumn{2}{| l ||}{$\bullet$ Moduli have no prime factors greater than $5$} \\

\hline

$10$ &  \multicolumn{2}{| l ||}{6267 equivalence classes, 1652 sets of moduli} \\ 

\ & \multicolumn{2}{| l ||}{$\bullet$ $1,2,4,6,8,12,$ or $18$ equivalence classes for each set of moduli} \\

\ & \multicolumn{2}{| l ||}{$\bullet$ All systems have minimum modulus $2$} \\

\ & \multicolumn{2}{| l ||}{$\bullet$ Maximum moduli: $30,40,45,48,60,72,80,90,96,108,120,144,160,192,216,288,384$} \\

\ & \multicolumn{2}{| l ||}{$\bullet$ Moduli have no prime factors greater than $5$} \\

\hline
\end{tabular} 

\end{table}

Table \ref{CSdata} completes the proof of Proposition \ref{mainprop}, since all of the classified covering systems have minimum modulus $2$, and \eqref{krukcs} is a distinct covering system with minimum modulus $3$ and exactly $11$ moduli. 

To go slightly further, we investigate the extent to which \eqref{krukcs} is unique in this regard. Again using Proposition \ref{lcmlist} and the aforementioned Jenkin-Simpson algorithm, the moduli of \eqref{krukcs} is the only list of $11$ moduli, all exceeding $2$, that survives the pruning process. From there, we run a brute force search with SageMath to classify all distinct covering systems with this specific list of moduli. Up to affine equivalence, we can translate any such system to contain $0(\text{mod }9)$ and $0(\text{mod }16)$. With appropriate scaling we can also fix the congruence $1(\text{mod }3)$, and by minimality no congruence class in the system can contain any other, which further narrows the search. The SageMath computation produces eight covering systems, four pairs related by multiplication by $7$, and we conclude with the following strengthening of Proposition \ref{mainprop}.
\begin{proposition} There are exactly four affine equivalence classes of distinct covering systems with at most $11$ moduli, all exceeding $2$, represented respectively by $$\{[1,3],[2,4],[5,6],[4,8],[0,9],[3,12],[0,16]\} \cup \begin{cases} \{[3,18],[0,24], [33,36], [8,48]\}, \\ \{[3,18], [8,24], [33, 36], [24,48] \},  \\ \{[15,18], [0,24], [21,36], [8,48] \}, \\ \{[15,18], [8,24], [21,36], [24,48] \}  \end{cases}.$$
\end{proposition}

%One glaring property of the data in Table \ref{CSdata} is that every distinct covering system with at most $10$ moduli has minimum modulus $2$. As mentioned in the introduction, Erd\H{o}s \cite{erdos2} conjectured that the system 
%\begin{equation*}\{[0,3], \ [0,4], \ [0,5], \ [1,6], \ [6,8], \ [3,10], \ [2,12], \ [11,15],  \ [7,20], \ [10,24], \ [2,30],  \ [34,40], \ [59,60], \ [98,120]\},
%\end{equation*}
%which has cardinality $14$, is smallest amongst distinct covering systems with all moduli exceeding $2$.  Again using Proposition \ref{lcmlist} and the aforementioned Jenkin-Simpson algorithm, the only candidate list of $11$ moduli that survives the pruning process is $$\{3,4,6,8,9,12,16,18,24,36,48\}.$$ In a twist, initially unbeknownst to the authors, a covering system with precisely this set of moduli, specifically \begin{equation*} \{[2, 3], \ [0, 4], \ [1, 6], \ [2, 8], \ [0, 9], \ [3, 12], \ [6, 16], \ [3, 18], \ [6, 24], \ [33, 36], \ [46, 48]\}, \end{equation*} appears in the unpublished Ph.D. thesis of Krukenberg \cite{Kruk}. With this revelation, the classification data in Table \ref{CSdata} assures that the minimum cardinality of a distinct covering system is all moduli exceeding $2$ is exactly $11$. 

\noindent \textbf{Acknowledgements:}  This research was initiated during the Summer 2022 Kinnaird Institute Research Experience at Millsaps College. All authors were supported during the summer by the Kinnaird Endowment, gifted to the Millsaps College Department of Mathematics. At the time of completion, all authors except Alex Rice were Millsaps College undergraduate students.


\begin{thebibliography}{10} 

\bibitem{BBMST} {\sc P. Balister, B. Bollobas, R. Morris, J. Sahasrabudhe, M. Tiba}, {\em On the Erd\H{os} covering problem: the density of the uncovered set}, Invent. Math. 228 (2022), 377-414.

\bibitem{DT} {\sc J. Dalton, O. Trifonov}, {\em Extreme covering systems}, J. Integer. Seq. 25 (2022), no. 9, Art. 22.9.1, 28pp.


\bibitem{erdos2} {\sc P. Erd\H{o}s,} {\em Egy kongruenciarendszerekrol sz\'ol\'o probl\'em\'ar\'ol (On a problem concerning congruence systems)}, in Hungarian, Mat. Lapok 4 (1952), 122-128. 

\bibitem{erdos} {\sc  P. Erd\H{o}s}, {\em On integers of the form $2^k+p$ and some related problems}, Summa Brasil. Math. 2 (1950), 113-123.

\bibitem{HHM} {\sc J. Hammer, J. Harrington, K. Marotta}, {\em Odd coverings of subsets of the integers}, J. Comb. Number Theory 10 (2018), 71-90.

\bibitem{harrington} {\sc J. Harrington}, {\em Two questions concerning covering systems}, Int. J. Number Theory 11 (2015), 1739-1750.

\bibitem{HSW} {\sc J. Harrington, Y. Sun, T.W.H. Wong}, {\em Covering systems with odd moduli}, Disc. Math. 345 (2022), no. 8, 112936, 12 pp.
 
\bibitem{hough} {\sc B. Hough}, {\em Solution of the minimum modulus problem for covering systems}, Annals of Mathematics, Second Series, vol. 181 (2015), no. 1, 361-382

\bibitem{JS} {\sc S. Jenkin, J. Simpson}, {\em Composite covering systems of minimal cardinality}, Integers vol. 3 (2003), paper A13.

\bibitem{JW} {\sc L. Jones, D. White}, {\em A group-theoretic approach to covering systems}, Algebra Discrete Math. 20 (2015), no. 2, 250-262.

\bibitem{Kruk} {\sc C. E. Krukenberg}, {\em Covering sets of the integers}, Ph.D. dissertation, University of Illinois, Urbana-Champaign, 1971. 

\bibitem{nielsen} {\sc P. Nielsen}, {\em A covering system whose minimum modulus is $40$}, Journal of Number Theory 129 (2009), 640-666.

\bibitem{Owens} {\sc T. Owens}, {\em A covering system with minimum modulus 42}, Masters thesis, Brigham Young University, 2014. Available at https://scholarsarchive.byu.edu/etd/4329.

\bibitem{PS} {\sc S. Porubsky, J. Schonheim}, {\em Covering systems of Paul Erd\H{os} past, present, and future}, Paul Erd\H{o}s and his Mathematics  I, Budapest, 2002.

\bibitem{zen} {\sc Z. W. Zen}, {\em Classified publications on covering systems}, \url{http://maths.nju.edu.cn/~zwsun/Cref.pdf}, updated 2006.

\end{thebibliography}
\end{document}